\documentclass{amsart}

\usepackage[T1]{fontenc}
\usepackage[utf8]{inputenc}
\usepackage{amsmath,amssymb,mathtools}
\usepackage{xcolor}
\usepackage{todonotes}

\newtheorem{theorem}{Theorem}[section]
\newtheorem{proposition}[theorem]{Proposition}
\newtheorem{lemma}[theorem]{Lemma}
\newtheorem{corollary}{Corollary}[section]

\newtheorem{smcorollary}[theorem]{Corollary}
\newcommand{\Var}{\operatorname{Var}}
\theoremstyle{remark}
\newtheorem*{remark}{Remark}
\newtheorem{question}{Question}
\newtheorem{definition}{Definition}
\newtheorem{problem}{Problem}

\newcommand{\F}{\mathbb{F}}
\newcommand{\Fp}{\mathbb F_p}

\newcommand{\PP}{\mathbb P}

\newcommand{\one}{\mathbf 1}
\newcommand{\Disc}{\operatorname{Disc}}

\title{Non-existence of sets with few special directions}
\author{Luca Ghidelli}
\thanks{The second and third authors have been unable to establish contact with Luca Ghidelli despite repeated attempts using his previously known email address and through professional contacts. His current contact information is therefore unknown to us. A substantial part of the work presented in this paper was carried out several years ago, and Luca Ghidelli made essential intellectual contributions to its development. In view of these contributions, we consider it appropriate to retain him as a co-author of the present manuscript.
 }
\author{Gergely Kiss}
\address[Gergely Kiss]{Corvinus University of Budapest, Department of Mathematics,
 Budapest, Hungary, and HUN-REN Alfr\'ed R\'enyi Institute of Mathematics,
 Budapest, Hungary}
\email{kigergo57@gmail.com}
\thanks{Gergely Kiss is supported by National Research, Development and
 Innovation Fund, OTKA grants no. FK 142993, Starting 150576 and Excellence 154121.}
\author{G\'abor Somlai}
\address[G\'abor Somlai]{The University of Melbourne, Melbourne, Australia, and E\"otv\"os Lor\'and University, Budapest, Hungary
 (on unpaid leave)}
\email{gabor.somlai@unimelb.edu.au, gabor.somlai@ttk.elte.hu}

\thanks{G\'abor Somlai was supported by  ARC Discovery Project DP250104965, and OTKA grant no. SNN 152582.}
\date{}

\begin{document}

\begin{abstract}
Let $p$ be an odd prime and let $S\subseteq\Fp^2$ have cardinality divisible by $p$. We prove that, for all sufficiently large primes $p$, no such set has exactly four special directions, and obtain a conditional extension to larger numbers of special directions under an affine-independence assumption on the corresponding projection functions. A separate second-moment argument shows more generally that, for every fixed $k\ge4$, no subset of $\Fp^2$ has exactly $k$ special directions once $p$ is sufficiently large. In fact, the result holds uniformly for $k$ up to a positive constant times $\sqrt p/\log p$. In contrast, for multisets every prescribed collection of at most $p$ directions can occur as the set of special directions of a $\{0,1,2\}$-valued multiset on $\Fp^2$.
\end{abstract}

\maketitle

\section{Introduction}

Let $p$ be an odd prime. For $0\ne u\in\Fp^2$ and $t\in\Fp$, write
\[
 L_{u,t}=\{x\in\Fp^2:\langle  u,x \rangle =t\},\qquad
 \pi_{u,S}(t)=|S\cap L_{u,t}|,
\]
where $\langle  u,x \rangle $ is the standard bilinear function over $\mathbb{F}_p$.  Here $L_{u,t}$ runs through the lines parallel to $u^\perp$, and
$\pi_{u,S}(t)$ counts the points of $S$ on $L_{u,t}$.
Throughout the paper we consider sets $S\subseteq\Fp^2$ whose cardinality
is divisible by $p$, say $|S|=np$.
The direction $u^\perp$ is called \emph{non-special} if
\[
 \pi_{u,S}(t)=n \qquad\text{for every }t\in\Fp,
\]
and \emph{special} otherwise. In other words, a direction is non-special
if $S$ is equidistributed among the lines in that direction, with exactly
$n$ points on each line.
For $|S|=p$, the special directions are exactly
those directions in the classical sense determined by pairs of distinct points of $S$, so this notion which was introduced in \cite{ghidelli} and \cite{kisssomlai} is a natural generalization of Rédei's direction problem.

The classical results for sets of at most $p$ points are the following. The
R\'edei--Megyesi theorem states that a $p$-point set in $\Fp^2$ is a line,
and hence determines one direction only, or determines at least
$(p+3)/2$ directions \cite{redei}. Dress, Klin and Muzychuk obtained the
same dichotomy independently and used it to give a proof of Burnside's
theorem for permutation groups of prime degree \cite{DKM}. Sz\H{o}nyi
proved that every non-collinear $m$-point set, $2\le m\le p$, determines at
least $\lceil(m+3)/2\rceil$ directions \cite{szonyi}. Lov\'asz and
Schrijver classified the equality case for $m=p$. A set of $p$ points determines
exactly $(p+3)/2$ directions if and only if it is affine-equivalent to the
graph of the function $x\mapsto x^{\frac{p+1}{2}}$ \cite{LS}. G\'acs \cite{gacs} proved that there is another gap in the number of directions. He proved that if a
$p$-point set determines more than $(p+3)/2$ directions, then it determines
at least $   \left\lfloor\frac{2(p-1)}{3}\right\rfloor+1$ directions.

We would like to highlight the connection of these results with the uncertainty principle. This comes from the following exact Fourier interpretation of special directions. For a function
$f\colon\Fp^2\to \mathbb{C}$
\begin{equation}\label{eq:defoffourier}
 \widehat f(\xi):=\sum_{x\in\Fp^2}f(x)e_p(-\langle \xi,x \rangle ), \mbox{ where}
 \quad e_p(t):=\exp(2\pi i t/p),
\end{equation}
denote the Fourier transform of the function $f$.
For the characteristic function $\one_S$ of a set $S$, one has
\[
 \widehat{\pi_{u,S}}(x)=\widehat{\one_S}(xu)\quad \mbox{for every  }x\in\Fp,
\]
and we will rely on this fact that 2 dimensional information is encoded in one dimensional language here.
Consequently, for $u\ne 0$ the direction $u^\perp$ is non-special if and only if
$$\widehat{\one_S}(xu)=0 \qquad\textrm{ for every } x\ne0.$$ Thus special
directions correspond, by taking orthogonal subspace, to one-dimensional subspaces meeting
$\operatorname{supp}\widehat{\one_S}\setminus\{0\}$. B\'ir\'o and Lev \cite{BL}
proved uncertainty (the support of a function or its Fourier transform is large) inequalities for functions on $\Fp^2$. Lev \cite{Lev}
deduced that a nonzero real-valued function supported on a set $S$ has at
most $|S|/2$  non-special (the author of \cite{Lev} called it perfect) directions, except when $S$ is a line and the function
is constant on it. For $f=\one_S$ with $p\mid |S|$, perfect
directions are precisely the non-special directions. In particular, for a
non-collinear $p$-point set there are at most $(p-1)/2$ perfect or equivalently  non-special
directions and hence at least $(p+3)/2$ special directions, recovering the
R\'edei--Megyesi bound.

For arbitrary cardinalities, Ghidelli \cite{ghidelli} calls a direction non-special when
every parallel line contains either $\lfloor |U|/p\rfloor$ or
$\lceil |U|/p\rceil$ points of $U$. If
$|U|=np-r$, where $1\le n\le p$ and $0\le r<p$, he proved that either $U$
is contained in the union of $n$ lines or $U$ has at least
$\left\lceil\frac{p+n+2-r}{n+1}\right\rceil$
special directions. When $r=0$, the first alternative means
that $U$ is the union of $n$ parallel lines, and the lower bound becomes
$\lceil(p+n+2)/(n+1)\rceil$.

Kiss and Somlai \cite{kisssomlai} proved that a set of cardinality divisible by $p$ cannot have
exactly two special directions, but the same result appears in a more Fourier analysis based paper of Fallon, Mayeli and Villano \cite{FMV}. Identifying $\Fp$ with
$\{0,1,\ldots,p-1\}$ and denoting
\[
 T_p=\{(a,b)\in\Fp^2:b<a\},
\]
they classified the sets with exactly three special directions. Every
set with exactly three special directions is an affine image of $T_p$ or of its complement.  The second ingredient from \cite{kisssomlai}, crucial here, is that every  projection polynomial 
of a set with $k$ special directions
has degree at most $k-2$ if $k \ge 3$. Finally, they constructed sets attaining
Ghidelli's lower bound with exactly four special directions for
$p=5,7,11$, and a $65$-point example in $\Fp^2$ for $p=13$.
Quite recently,
it was shown \cite{Wang} that there is no set of $52$-points with four special directions in $\mathbb{F}_{13}^2$ answering a question of \cite{kisssomlai}. 
Adriaensen, Sz\H{o}nyi and Weiner \cite{ASW} extended the concept of the
projection-function to multisets in $\operatorname{AG}(2,q)$,
$q=p^h$, in a modular setting. In particular, they proved the analogous
$k-2$ degree bound for mod-special projection functions, where $k$ denotes the number of mod-special directions, and obtained lower
bounds for the numbers of mod-special and ordinary special directions over
extension fields. They also explicitly asked the following question  \cite[Section~7]{ASW}.

\begin{question}\label{q1}
Does it hold that for each $k \ge 4$ and all sufficiently large primes $p$ (where the lower bound on $p$ depends on $k$), $\operatorname{AG}(2,p)$ has no sets of points with exactly $k$ special directions?
\end{question}

In this paper, we answer this question in full generality, but first, we give an affirmative answer for the case $k=4$, which is proved by our first method based on discrepancy estimates. This can be formulated as follows.
\begin{theorem}\label{thm:4directions}
There exists $T_4 \in \mathbb{N}$ such that, for every prime $p\ge T_4$, no
subset of $\Fp^2$ has exactly four special directions.
    \end{theorem}
In Section \ref{sec:equidistribution} we collect knowledge on discrepancy bounds. We prove Theorem \ref{thm:explicit diszkrepancia masodfokuakra} using standard methods to obtain a completely explicit discrepancy bound with respect to polynomials of small degree and collect some of the main results on discrepancy we use later in Section \ref{sec:affineextension}.

In Section \ref{sec:projection} we show how a set can be reconstructed from its special projection functions.
The proof of Theorem~\ref{thm:4directions} contained in Section \ref{sec:k4} is based on two main ingredients.
First, the projection polynomial theorem of Kiss and Somlai \cite{kisssomlai},
stated here as Proposition~\ref{projections:degree}, shows that if $S$
has $k$ special directions, then all of its projection polynomials have
degree at most $k-2$. In the case $k=4$, this reduces the relevant
projection functions to polynomials of degree at most two. The proof is a combination of the
affine-independence statement of Lemma~\ref{lem:3independent} with the
quadratic exponential-sum estimate of Lemma~\ref{lem:2.1C} and the discrepancy
estimate of Proposition~\ref{prop: diszkrepancia szamitas}.

The same method also gives the following conditional extension to larger fixed
numbers of special directions.
\begin{theorem}[Conditional extension]\label{thm:affineextension-intro}
Fix integers $k\ge4$ and $r$ such that
\[
 \frac{k+1}{2}<r\le k-1.
\]
Then there exists $T_{k,r}$ such that, for every prime $p\ge T_{k,r}$,
there is no subset $S\subseteq\Fp^2$ of cardinality divisible by $p$
with exactly $k$ special directions for which $r$ of the corresponding
auxiliary projection functions, regarded as polynomials on $\Fp^2$, are
affinely independent.
\end{theorem}
Theorem \ref{thm:affineextension-intro} is proved in Section~\ref{sec:affineextension}. Here the quadratic
estimate is replaced by the bounded-degree exponential-sum estimate of
Corollary~\ref{cor:weyl}, which follows from the results of Deligne and Katz.

A separate second-moment argument answers Question \ref{q1} affirmatively in Section \ref{sec:second-moment}. This gives a conclusion stronger than that of Theorem \ref{thm:4directions} and Theorem \ref{thm:affineextension-intro} using a second-moment argument.
We present Theorem~\ref{thm:fixed-k-intro} and its proof separately from the preceding results because, although its basic idea and main mathematical ingredients are due to the authors, its detailed implementation was carried out by generative AI (ChatGPT 5.6 (OpenAI)), with Proposition~\ref{prop:poly-variance} being the main AI-produced calculation. 
All other mathematical work in the paper is entirely the authors' own.

\begin{theorem}\label{thm:fixed-k-intro}
Let $p$ be a prime and let $k$ be an integer greater than $3$.
There is no subset $S\subseteq\Fp^2$ having exactly $k$ special directions if
$$
 4\le k\le \frac18\frac{\sqrt p}{2+\log p}.
$$
\end{theorem}
\begin{remark}
The explicit bound in Theorem~\ref{thm:fixed-k-intro} is not intended to be sharp.
For $k=4$, the displayed condition first becomes non-vacuous at
$p=207766$; hence the first prime to which it applies is $p=207769$.
\end{remark}
The quantitative bounds obtained above are far from expected to be sharp.
This leads to the following natural problem.

\begin{question}
For a fixed integer $k\ge4$, determine the smallest number $T_k$ such that,
for every prime $p>T_k$, there is no subset of $\Fp^2$ having exactly
$k$ special directions.
\end{question}

For $k=4$, examples are known for $p=5,7,11$ and $13$.
Preliminary computations suggest that no example exists for the primes
between $17$ and $100$. This motivates the following more specific question.

\begin{question}
Is there a subset of $\Fp^2$ with exactly four special directions for any
prime $p\ge17$?
Equivalently, is the optimal threshold $T_4=13$?
\end{question}

Finally, in Section~\ref{sec:multisets} we show that the situation is completely different for multisets. We prove Theorem~
\ref{thm:prescribed-multiset-directions} which states that for every odd prime
$p$ and every prescribed set
$\mathcal D\subseteq\PP^1(\Fp)$ with $|\mathcal D|\le p$, there exists a
$\{0,1,2\}$-valued multiset on $\Fp^2$ whose special directions are
exactly the elements of $\mathcal D$. The construction is based on the family of lines
which may be viewed as the tangent lines to the parabola
$y=-x^2/4$.

\section{Equidistribution and exponential-sum estimates}
\label{sec:equidistribution}

In this section we recall the analytic and number-theoretic tools that will be
used later in the proof. Our approach relies on quantitative forms of
equidistribution modulo a prime $p$, measured by the discrepancy of finite
point sets, and on exponential-sum estimates for polynomial functions.


In particular, we will use that bounded-degree polynomial maps whose
nontrivial linear combinations are nonconstant have small interval-box
discrepancy when $p$ is large. This will allow us to find points at which
several projection functions are simultaneously very small or very large.

For the second-moment argument in Section~\ref{sec:second-moment}, we also
record the classical one-variable Weil estimate.

\begin{theorem}[Weil \cite{Weil}]\label{thm:weil}
Let $P\in\Fp[x]$ be a nonconstant polynomial of degree $d<p$. Then for every
$h\in\Fp^\times$,
\[
 \left|\sum_{t\in\Fp}e_p(hP(t))\right|\le(d-1)\sqrt p.
\]
For $d=1$ the sum is $0$.
\end{theorem}
\subsection{Discrepancy bounds}
Discrepancy, which is our main tool in the proof of Theorem \ref{thm:4directions}, provides a quantitative way to measure how far a
finite sequence deviates from the uniform distribution.
Since the second and third authors had limited knowledge of this tool, and we believe this widely used tool could have other interesting applications in the discrete geometry community, we gathered a broader set of references.
For general background on uniform distribution and discrepancy theory, we
refer to the classical monograph of Kuipers and Niederreiter
\cite{KuipersNiederreiter}, the survey volume edited by Granville and Rudnick
\cite{GranvilleRudnick}, and the modern account of Niederreiter
\cite{neider}. The presentation
below follows the classical works of Koksma \cite{koksma} and Niederreiter
\cite{neider}, see also the exposition in \cite{BL}.

Let
\[
 \Gamma=\{\gamma_j=(\gamma_{1,j},\ldots,\gamma_{r,j})\in[0,1)^r:
 1\le j\le N\}
\]
be a finite sequence of $N$ points in the $r$-dimensional unit cube. For a box
\[
 B=[a_1,b_1)\times\cdots\times[a_r,b_r)\subset[0,1)^r,
\]
define the relative discrepancy of $\Gamma$ with respect to $B$ by
\[
 \Disc_\Gamma(B)=
 \left|\frac{\#(B\cap\Gamma)}{N}-\prod_{i=1}^r(b_i-a_i)\right|,
\]
and the absolute discrepancy by
\[
 \Delta_\Gamma=\sup_B\Disc_\Gamma(B),
\]
where the supremum is taken over all boxes $B\subset[0,1)^r$.

For $\mathbf a=(a_1,\ldots,a_r) \in \mathbb{Z}^r$
let $e(z)=e^{2\pi iz}$ denote the
normalized exponential function and set
\[
 S_{\mathbf a}(\Gamma)=\sum_{j=1}^N e(-\mathbf a\cdot\gamma_j),
\]
where $\mathbf a\cdot\gamma_j$ denotes the usual scalar product of vectors in
$\mathbb{R}^r$. A standard form of the Erd\H{o}s--Tur\'an--Koksma
\cite{koksma,neider} inequality gives the following estimate.

\begin{theorem}[Erd\H{o}s--Tur\'an--Koksma inequality]\label{thm:koksma}
Let $\Gamma\subset[0,1)^r$ be as above. Then for every integer $H\ge1$,
\[
 \Delta_\Gamma\le c_r\left(
 \frac1H+\frac1N
 \sum_{0<\|\mathbf a\|_\infty\le H}
 \frac{|S_{\mathbf a}(\Gamma)|}
 {\prod_{i=1}^r(1+|a_i|)}
 \right),
\]
where $c_r>0$ is an absolute constant depending only on $r$, and
$\|\cdot\|_\infty$ denotes the supnorm.
\end{theorem}

Thus equidistribution can be deduced from suitable estimates for exponential
sums.


\subsubsection{A direct discrepancy estimate for cartesian product of interavls}

To prove Proposition~\ref{prp:4directions-equid} we do not need the full
Erd\H{o}s--Tur\'an--Koksma inequality. Since the 'boxes' used in this inequality are
Cartesian products of intervals, Fourier inversion gives a direct estimate
with an explicit constant, that we believe is better than the general constant for $r=3$. In the remaining part of this section we derive discrepancy bounds in this specific setting. The argument is standard, we will separate the trivial representation and estimate the Fourier coefficients corresponding to the nontrivial representation to obtain the required estimate on the discrepancy.

Let us identify $\Fp$ with $\{0,1,\ldots,p-1\}$ and call a set of consecutive
residues an interval. Let
\[
 e_p(t):=\exp\!\left(\frac{2\pi i t}{p}\right)\qquad \text{and } \qquad
 \widehat f(\xi):=\sum_{z \in \Fp^r}f(z)e_p(-\langle  \xi , z  \rangle ).
\]
Let $ H_q=\sum_{h=1}^{q}\frac{1}{h} $.
For $q=(p-1)/2$, let
$ L_p=1+H_q$.
\begin{lemma}\label{lem:fourier egyutthato becsles}
    If $I\subseteq\Fp$ is an interval, then
\[
 \bigl\|\widehat{\one_I}\bigr\|_1
 =\sum_{h\in\Fp}\bigl|\widehat{\one_I}(h)\bigr|
 \le pL_p\le p(2+\log p).
\]
\end{lemma}

\begin{proof}
Let $m=|I|$. Since translation of $I$ only multiplies each Fourier coefficient by a complex number of absolute value $1$, it does not change the absolute values of the Fourier coefficients. Thus we may assume that $ I=\{0,1,\ldots ,m-1\}$.
For $h=0$, we simply have $ \widehat{\one_I}(0)=m$, so $|\widehat{\one_I}(0) |=m$.
Now let $h\ne0$. Since $e_p(-h)\ne1$, the formula for the sum of the elements of a finite geometric series gives
\[ \widehat{\one_I}(h) =\sum_{j=0}^{m-1}e_p(-hj) =\frac{1-e_p(-hm)}{1-e_p(-h)}. \]
Consequently
$ \left| \widehat{\one_I}(h)\right| =\frac{|1-e_p(-hm)|}{|1-e_p(-h)|}.$
We now calculate the two absolute values explicitly. For every real number $x$,
\[ 1-e^{-2\pi i x} =e^{-\pi i x}\bigl(e^{\pi i x}-e^{-\pi i x}\bigr) =2i\,e^{-\pi i x}\sin(\pi x), \]
and therefore
\[ |1-e^{-2\pi i x}| =2|\sin(\pi x)|. \]
Applying this identity first with $x=hm/p$ and then with $x=h/p$, we obtain
\[ |1-e_p(-hm)| =2\left|\sin\!\left(\frac{\pi hm}{p}\right)\right|
\text{ and} \quad |1-e_p(-h)| =2\left|\sin\!\left(\frac{\pi h}{p}\right)\right|. \]
Hence, for every $h\ne0$, \[ \bigl|\widehat{\one_I}(h)\bigr| = \frac{ \left|\sin\!\left(\frac{\pi hm}{p}\right)\right| }{ \left|\sin\!\left(\frac{\pi h}{p}\right)\right| } \le \frac{1}{ \left|\sin\!\left(\frac{\pi h}{p}\right)\right| } \] since the numerator is at most $1$

Since $\sin(\pi h/p)=\sin(\pi(p-h)/p)$ and
$\sin(\pi h/p)\ge 2h/p$ for $1\le h\le q$, we obtain
\[
 \sum_{h=1}^{p-1}\frac1{|\sin(\pi h/p)|}
 \le2\sum_{h=1}^{q}\frac{p}{2h}=pH_q.
\]
The term at $h=0$ is $m\le p$, proving the first inequality of the statement of this lemma. The second statement
follows from the inequality $H_q\le1+\log q$.
\end{proof}

\begin{proposition}\label{prop: diszkrepancia szamitas}
    Let $X$ be a finite set of cardinality $N$ and let
$G\colon X \to \F_p^r$, and suppose that
\[
 \left|\sum_{x\in X}e_p(\langle  \xi,G(x)  \rangle )\right|\le B
 \qquad   \text{whenever } 0 \ne \xi \in \F_p^r.
\]
Then, for every interval box $T=I_1\times\cdots\times I_r$,
\[
 \left|\#\{x\in X:G(x)\in T\}-\frac{N|T|}{p^r}\right|
 \le BL_p^r.
\]
Equivalently, its relative interval-box discrepancy is at most
$(B/N)L_p^r$.
\end{proposition}
\begin{proof}
We apply Fourier inversion formula to the indicator function
\[
 \one_T \colon \F_p^r \to \{ 0 ,1 \}
 \qquad \text{and  } \quad
 \one_T(y)=
 \begin{cases}
  1, & y \in T,\\
  0, & y \notin T.
 \end{cases}
\]
With our convention on the normalization of the Fourier transform and using the notation introduced in equation \eqref{eq:defoffourier}
the Fourier inversion formula is
\[
 f(y)
 =
 \frac1{p^r}
 \sum_{\xi\in \F_p^r}
 \widehat f(\xi)e_p(\langle \xi,y \rangle ).
\]
Applying this to $f=\one_T$ gives, for every
$y\in \F_p^r$,
\[
 \one_T(y)
 =
 \frac1{p^r}
 \sum_{\xi\in \F_p^r}
 \widehat{\one_T}(\xi)
 e_p(\langle \xi,y \rangle ).
\]
In particular, substituting $y=G(x)$, we obtain
\begin{equation}\label{eq:indikator fuggveny inverz fourievel}
\one_T(G(x))
 =
 \frac1{p^r}
 \sum_{\xi\in \F_p^r}
 \widehat{\one_T}(\xi)
 e_p(\langle \xi,G(x) \rangle ).
\end{equation}
We now sum equation \eqref{eq:indikator fuggveny inverz fourievel} over $x\in X$. The left-hand side of equation \eqref{eq:indikator fuggveny inverz fourievel} is equal to
$ \sum_{x\in X}\one_T(G(x))$.
By the definition of the indicator function, each $x\in X$ contributes
$1$ to this sum exactly when $G(x)\in T$, and contributes $0$
otherwise. Hence
\[
 \sum_{x\in X}\one_T(G(x))
 =
 \#\{x\in X:G(x)\in T\}.
\]
In equation \eqref{eq:indikator fuggveny inverz fourievel}, the term on the right-hand side corresponding to $\xi=0$ is $\frac{N|T|}{p^r}$.
Consequently, after separating the trivial representation and changing the order of the summation, we have
\[
 \#\{x\in X:G(x)\in T\} - \frac{N|T|}{p^r} =
 \frac{1}{p^r}
 \sum_{\substack{\xi\in \F_p^r\\ \xi\ne0}}
 \widehat{\one_T}(\xi)
 \sum_{x\in X}e_p(\langle  \xi,G(x)  \rangle ).
\]
We now take absolute values of both sides of this equation. By the triangle inequality,
\begin{align*}
 \left|
 \#\{x\in X \colon G(x)\in T\}
 -
 \frac{N|T|}{p^r}
 \right|
 \le
 \frac1{p^r}
 \sum_{\substack{\xi\in \F_p^r\\ \xi\ne0}}
 \left|\widehat{\one_T}(\xi)\right|
 \left|
 \sum_{x\in X}e_p(\langle \xi,G(x) \rangle )
 \right|.
\end{align*}
By the assumption of the proposition, for every nonzero
$\xi\in \mathbb{F}_p^r$,
we may bound each term in the preceding sum by $B$. This gives us
\begin{equation*}
 \left|
 \#\{x\in X:G(x)\in T\} - \frac{N|T|}{p^r}
 \right|\le
 \frac{B}{p^r}
 \sum_{\substack{\xi\in \F_p^r\\ \xi\ne0}}
 \left|\widehat{\one_T}(\xi)\right|.
\end{equation*}
Since all terms in the last sum are nonnegative, adding the term
corresponding to $\xi=0$ can only increase the sum. Hence
$
 \sum_{\substack{\xi\in \F_p^r\\ \xi\ne0}}
 \left|\widehat{\one_T}(\xi)\right|
 \le
 \sum_{\xi\in \F_p^r}
 \left|\widehat{\one_T}(\xi)\right|$,
and consequently
\begin{equation}\label{eq:discrepancy trivial repr separation}
     \left|
 \#\{x\in X:G(x)\in T\}
 -
 \frac{N|T|}{p^r}
 \right|
 \le
 \frac{B}{p^r}
 \sum_{\xi\in \F_p^r}
 \left|\widehat{\one_T}(\xi)\right|.
\end{equation}

We have already seen above that since
$T=I_1\times\cdots\times I_r$ is Cartesian product we have
\[
 \widehat{\one_T}(\xi)
 =
 \prod_{j=1}^r\widehat{\one_{I_j}}(\xi_j)
 \qquad
 \text{for any representation } \xi=(\xi_1,\ldots ,\xi_r).
\]
It follows that
\begin{equation}\label{eq:nemsplitL1becsles}
\begin{split}
     \sum_{\xi\in \F_p^r}
 \left|\widehat{\one_T}(\xi)\right|
 &=
 \sum_{\xi_1,\ldots ,\xi_r\in\mathbb F_p}
 \prod_{j=1}^r
 \left|\widehat{\one_{I_j}}(\xi_j)\right|=
 \prod_{j=1}^r
 \sum_{\xi_j\in\mathbb F_p}
 \left|\widehat{\one_{I_j}}(\xi_j)\right|
 =
 \prod_{j=1}^r
 \bigl\| \widehat{\one_{I_j}} \bigr\|_1.
\end{split}
\end{equation}
By Lemma \ref{lem:fourier egyutthato becsles} we have
$
 \bigl\|\widehat{\one_{I_j}}\bigr\|_1
 \le pL_p$
 for j=1,\ldots, r.
Therefore by equation \eqref{eq:nemsplitL1becsles}
\[
 \sum_{\xi\in \F_p^r}
 \left|\widehat{\one_T}(\xi)\right|
 =
 \prod_{j=1}^r
 \bigl\|\widehat{\one_{I_j}}\bigr\|_1
 \le
 \prod_{j=1}^r pL_p
 =
 (pL_p)^r.
\]
Substituting this estimate into equation \eqref{eq:discrepancy trivial repr separation} gives
\begin{align*}
 \left|
 \#\{x\in X:G(x)\in T\}
 -
 \frac{N|T|}{p^r}
 \right|
 \le
 \frac{B}{p^r}(pL_p)^r
 =BL_p^r.
\end{align*}
This proves the claimed upper bound. Dividing by $N$ gives the relative discrepancy bound $\frac{B}{N}L_p^r$.
\end{proof}
\begin{lemma}\label{lem:2.1C}
    Let $p$ be odd and let $Q\in\Fp[x,y]$ be nonconstant of degree at most $2$ and let
$S(Q):=\sum_{(x,y)\in\Fp^2}e_p(Q(x,y))$. Then $|S(Q)| \le p^{3/2}$.

If $Q$ is linear, them $S(Q)$ is zero.
\end{lemma}
\begin{proof}
If $Q$ is nonconstant linear, summing first in a variable with a nonzero
coefficient gives $S(Q)=0$.

Suppose now that $Q$ has nonzero quadratic part. Write
\[
 Q(x_1,x_2)=Q_2(x_1,x_2)+Q_1(x_1,x_2)+Q_0,
\]
where $Q_2$ is homogeneous of degree two, $Q_1$ is linear, and
$Q_0\in\mathbb F_p$ is constant.
The constant term $Q_0$ only contributes the factor $e_p(Q_0)$,
which has absolute value one, so it does not affect $|S(Q)|$.

Since $p$ is odd, after an invertible linear change of variables the
quadratic form $Q_2$ can be diagonalized. As an invertible change of
variables permutes the elements of $\mathbb F_p^2$, it does not change
the value of the complete sum. Thus it is enough to consider two cases depending on the possible ranks of
$Q_2$.

First suppose that $\operatorname{rank}(Q_2)=2$. Then, after a linear
change of variables,
\[
 Q(u,v)=au^2+bv^2+cu+dv+e,
 \qquad a,b\ne0.
\]
Hence
\begin{equation}\label{eq:splitS(Q)}
 \begin{split}
      S(Q)= e_p(e)
 \left(\sum_{u\in\mathbb F_p}e_p(au^2+cu)\right)
 \left(\sum_{v\in\mathbb F_p}e_p(bv^2+dv)\right).
 \end{split}
\end{equation}
The linear terms can be removed by completing the square
since $a \ne 0$ and $b \ne 0$:
and therefore
\begin{align*}
 \sum_{u\in\mathbb F_p}e_p(au^2+cu)
 &=
 e_p\left(-\frac{c^2}{4a}\right)
 \sum_{u\in\mathbb F_p}
 e_p\left(a\left(u+\frac{c}{2a}\right)^2\right) =
 e_p\left(-\frac{c^2}{4a}\right)
 \sum_{u\in\mathbb F_p}e_p(au^2).
\end{align*}
In the last equality we used that translation by $c/(2a)$ is a
bijection of $\mathbb F_p$. Thus
\begin{equation}\label{eq:gyokp1}
     \left|
 \sum_{u\in\mathbb F_p}e_p(au^2+cu)
 \right|
 =
 \left|
 \sum_{u\in\mathbb F_p}e_p \left(au^2 \right)
 \right|=\sqrt p,
\end{equation}
since the sum on the right is a  Gauss sum.
The same argument gives
\begin{equation}\label{eq:gyokp2}
     \left|
 \sum_{v\in\mathbb F_p}e_p(bv^2+dv)
 \right|
 =\sqrt p.
\end{equation}
and then it follows from equations \eqref{eq:splitS(Q)}, \eqref{eq:gyokp1} and \eqref{eq:gyokp2} that
\begin{equation}\label{eq:S(Q) becsles rank 2}
|S(Q)| = \sqrt{p} \,\sqrt{p} =p.
\end{equation}
Now suppose that $\operatorname{rank}(Q_2)=1$, which means that it has a nontrivial kernel. After an invertible
linear change of variables we may write
\[
 Q(u,v)=au^2+cu+dv+e,\]
where  $a \ne 0$ and
where $v$ is the null direction of the quadratic part corresponding to the kernel of $Q_2$. The complete
sum of equation \eqref{eq:splitS(Q)} factors as
\begin{align*}
 S(Q) =
 e_p(e)
 \left(\sum_{u\in\mathbb F_p}e_p(au^2+cu)\right)
 \left(\sum_{v\in\mathbb F_p}e_p(dv)\right).
\end{align*}
Since the sum of all $p$-th roots of unity is 0 we have
\[
 \sum_{v \in \F_p}e_p(dv) =
 \begin{cases}
  p,& d=0,\\
  0,& d \ne 0.
 \end{cases}
\]
Therefore, in this case,
\[
 |S(Q)|
 =
 \begin{cases}
  p\sqrt p=p^{3/2},& d=0,\\
  0,& d\ne0.
 \end{cases}
\]
In particular,
\begin{equation}\label{eq:S(Q) rank 1}
 |S(Q)|\le p^{3/2}.
\end{equation}
Combining the two cases, equations \eqref{eq:S(Q) becsles rank 2} and \eqref{eq:S(Q) rank 1} we obtain
\[
 \left|
 \sum_{x_1,x_2\in\mathbb F_p}e_p(Q(x_1,x_2))
 \right|
 \le p^{3/2}.
\]
%
\end{proof}
As a corollary we obtain the following theorem.
\begin{theorem}\label{thm:explicit diszkrepancia masodfokuakra}
Let
$G=(G_1,G_2,G_3)\colon \Fp^2 \to \Fp^3$. Assume that every nonzero
linear combination of $G_1,G_2,G_3$ is a nonconstant polynomial of degree at most 2.
Then
\begin{equation}\label{equation:diszkrepancia becsles kiszamolva}
     \sup_{T=I_1\times I_2\times I_3}
 \left|\frac{\#G^{-1}(T)}{p^2}-\frac{|T|}{p^3}\right|
 \le\frac{L_p^3}{\sqrt p}
 \le\frac{(2+ \log{p})^3}{\sqrt{p} }.
\end{equation}
\end{theorem}
The advantage of this theorem is that there is no unspecified discrepancy constant in equation \eqref{equation:diszkrepancia becsles kiszamolva}.
\subsection{Exponential-sum estimates}
 We next recall the exponential-sum bounds that will be used below. The sharp classical estimate of Deligne applies when the
highest-degree homogeneous part of the polynomial is nonsingular.
\begin{theorem}[Deligne \cite{deligne}]\label{thm:Deligne}
Let $f\in\Fp[x_1,\ldots,x_r]$ be a polynomial of total degree $d$ with
$p\nmid d$ and let $f_d$ denote its homogeneous part of degree $d$. If the projective hypersurface $\{f_d=0\}\subset\PP_{\Fp}^{r-1}$ is
nonsingular, then
\[
 \left|\sum_{x\in\Fp^r} e_p(f(x))\right|
 \le(d-1)^r p^{r/2}.
\]
\end{theorem}


 For singular highest-degree parts we use the following affine-space specialization of Katz's estimate for singular exponential sums, see also Cluckers--Nguyen \cite{CN, N}.

\begin{theorem}[Katz, Theorem 4 of {\cite{Ka}}]\label{thm:CNK}
Let $f\in\Fp[x_1,\ldots,x_r]$ be a nonconstant polynomial of degree $d$ with $p\nmid d$. Let
$s=s(f)$ be the dimension of the singular locus of projective hypersurface $\{f_d=0\}\subset\PP_{\Fp}^{r-1}$, with the convention that $s=-1$ if the singular locus is empty. Then
\[
 \left|\sum_{x\in\Fp^r} e_p(f(x))\right|
 \le C_{r,d}p^{(r+1+s)/2},
\]
where $C_{r,d}>0$ depends only on $r$ and $d$.
\end{theorem}


When $s=-1$, that is, when the projective hypersurface defined by $f_d$
is nonsingular, Katz's estimate recovers Deligne's bound up to the value
of the constant. In general, since $s\le r-2$, it gives the following
uniform consequence.

\begin{corollary}\label{cor:weyl}
For every fixed $r,D\ge 1$ there is a constant $C_{r,D}>0$ such that whenever $p>D$ and $f\in\Fp[x_1,\ldots,x_r]$ is a nonconstant polynomial of degree at most $D$
\[
 \left|\sum_{x\in\Fp^r}e_p(f(x))\right|
 \le C_{r,D}p^{r-1/2},
\]
where $C_{r,D}$ depends only on $r$ and $D$.
\end{corollary}

\begin{remark}
The four-direction argument below only requires the explicit quadratic
estimate of Lemma~\ref{lem:2.1C}, which yields the explicit interval-box discrepancy
bound in Subsection~2.1.1. Corollary~\ref{cor:weyl} will instead be used
in Section~\ref{sec:affineextension} for the conditional extension to
affinely independent families of projection functions of arbitrary bounded
degree.
\end{remark}

\section{Projection functions}\label{sec:projection}
In this section we investigate how the elements of a finite set
$S\subset\Fp^2$ are distributed along parallel lines, and how this
distribution can be described algebraically.

\subsection{Definition and basic properties}

For each direction $\mathbf m=(m_1,m_2)\in\PP^1(\Fp)$ consider the family of
parallel lines
\[
 \ell_{a,\mathbf m}=\{(x,y)\in\Fp^2:m_1x+m_2y=a\},\qquad a\in\Fp.
\]
There are exactly $p$ such lines in each direction. For a given subset
$S\subset\Fp^2$ we define the \emph{projection function} (or counting
function) in direction $\mathbf m$ by
\[
 \pi_{\mathbf m,S}(a)=\#(S\cap\ell_{a,\mathbf m}),\qquad a\in\Fp.
\]
Thus $\pi_{\mathbf m,S}\colon\Fp\to\{0,1,\ldots,p\}$ records the number of
points of $S$ on each line parallel to $\mathbf m$. We call
$\pi_{\mathbf m,S}$ \emph{equidistributed} if it is constant. Equivalently,
$\mathbf m$ is an equidistributed or non-special direction for $S$ if all
parallel lines contain the same number of points of $S$.

If $\pi_{\mathbf m,S}(a)=p$ for some $a$, then $S$ contains the entire line $L=\ell_{a,\mathbf m}$. For any direction different from that of $L$, removing $L$ decreases every projection value by exactly one, and therefore does not change whether that direction is special. In the direction of $L$, the type of the direction can only change in the degenerate cases $S=L$ or $S=\Fp^2$. Since below we are interested in sets with four special directions, neither case can occur. We may therefore remove complete (parallel) 
lines repeatedly without changing the set of special directions and assume from now on that
\[
 0 \le \pi_{\mathbf m,S}(a)\le p-1
 \quad \text{for all} \quad a\in\Fp \text{ and all directions } \mathbf{m}.
\]
Notice that this operation also preserves the divisibility of $|S|$ by $p$.

Now we may assume that no complete line is contained in $S$ and hence the projection functions take values in $\{0,1,\ldots, p-1\}$, which can be naturally identified with $\F_p$. Using these identifications, the projection functions can be expressed as polynomials in $\Fp$.

Over the finite field $\Fp$ we have $x^p=x$ for every $x\in\Fp$, hence every function
$\Fp\to\Fp$ can be represented uniquely by a polynomial of degree at most
$p-1$. Under this identification the projection functions are viewed as
polynomials of degree at most $p-1$, which we call the \emph{projection
polynomials} of $S$.


Invertible linear transformations act transitively on directions in $\Fp^2$ and preserve the number of special directions. They also permute the Fourier coefficients of the characteristic function, and in particular preserve the size and the 'linear structure' of the support of its Fourier transform. Hence all arguments below may be performed after applying a convenient invertible linear map.

Let $g_{\mathbf m}$ be an invertible linear transformation sending the
direction vector $\mathbf m=(m_1,m_2)$ to $(0,1)$. Then
$\pi_{\mathbf m,S}(u)=\pi_{1,Sg_{\mathbf m}}(u)$, so all projection functions
can be reduced to the vertical case by an affine change of coordinates.

\subsection{Degree of projection polynomial and number of special directions}

It follows from the results of Section~2 that a polynomial of small degree
over $\Fp$ gives rise to a distribution close to uniform. Hence the degree of
a projection polynomial may be viewed as a quantitative measure of how far
the corresponding direction is from being equidistributed. The second and third authors
\cite{kisssomlai} observed that this degree is closely related to the number
of special directions.

\begin{proposition}[{\cite{kisssomlai}}]\label{projections:degree}
Let $S\subseteq\Fp^2$ be a set with $k\ge2$ special directions. Then every
projection polynomial of $S$ has degree at most $k-2$.
\end{proposition}

\subsection{Representation by projection functions}

The set $S$ is completely determined by the Fourier transform of its
characteristic function, which in turn can be expressed in terms of the
projection functions. For this we recall another result of
\cite{kisssomlai}.

\begin{proposition}[{\cite{kisssomlai}}]\label{projections:combination}
Let $S\subseteq\Fp^2$ be a set with $d$ special directions. Then the
characteristic function $\one_S$ is a rational linear combination of
characteristic functions of lines parallel to the special directions of $S$.
\end{proposition}

To make this explicit, let
$\mathbf m_i=(m_{i,x},m_{i,y})$ for $i=1,\ldots,k$ denote the special
directions of $S$. Define the projection functions
\begin{equation}\label{eq:a_i}
 a_i(k):=\#\{(x,y)\in S:m_{i,x}x+m_{i,y}y=k\},\qquad k\in\Fp,
\end{equation}
and the auxiliary functions
\[
 S_i(x,y)=a_i(m_{i,x}x+m_{i,y}y),
\]
which are linear combinations of the characteristic functions of lines
parallel to $\mathbf m_i$. The coefficients $a_i(k)$ are nonnegative integers
between $0$ and $p$.

 From now on, by a slight abuse of notation, we also write $S$ for the characteristic function $\one_S$ of the set $S$.
In the following proposition, the functions $S$ and $S_i$ are treated as integer-valued functions, using the identification between $\mathbb{F}_p$ and the sets $\{0,1,\ldots, p-1\}$ again.	

\begin{proposition}\label{projections:explicit}
Let $S\subseteq\Fp^2$ be a set of cardinality $np$ having $k$ special
directions, and let $S_i$ be defined as above. Then
\begin{equation}\label{eq4.3}
 S(x,y)=-\frac{n(k-1)}{p}+\frac1p\sum_{i=1}^k S_i(x,y).
\end{equation}
\end{proposition}

\begin{proof}

Fix $z=(x,y)\in\Fp^2$. Consider the $p+1$ affine lines through $z$,
one in each direction. Summing the cardinalities of their intersections
with $S$, every point of $S\setminus\{z\}$ is counted exactly once, while
$z$, if it belongs to $S$, is counted $p+1$ times. Hence the sum over all
directions is
$$
np+pS(z).
$$
There are $p+1-k$ non-special directions, and the line through $z$ in each
such direction contains exactly $n$ points of $S$. Therefore their total
contribution is $(p+1-k)n$. Consequently,
$$
\sum_{i=1}^k S_i(z)
=np+pS(z)-(p+1-k)n
=(k-1)n+pS(z),
$$
which is equivalent to equation \eqref{eq4.3}.
\end{proof}

\subsection{Affine independence of projection functions}

It was proved in \cite{kisssomlai} that a subset of $\Fp^2$ having exactly
three special directions can be obtained, up to an invertible affine
transformation, from the so--called triangle or from its complement. The
triangle can be described as
\[
 \mathcal T_p=\{(x,y)\in\Fp^2:y\in\{0,1,\ldots,x-1\}\}.
\]
The proof in \cite{kisssomlai} proceeds in two steps. First,
Proposition~\ref{projections:degree} implies that when a set has exactly three
special directions, the corresponding projection polynomials are linear. In
the second step, one uses the following lemma, which although not stated
explicitly in \cite{kisssomlai}, is contained in its arguments.
For the sake of completeness, we present a short proof of it.

\begin{lemma}\label{lem:2linearpolynomial}
Suppose that no complete affine line is contained in $S$ and two special directions of $S\subset\Fp^2$ have linear
projection polynomials. Then $S$ is an affine
image of $\mathcal T_p$. In particular $S$ has exactly three special directions.
\end{lemma}
\begin{proof}
After an affine change of coordinates, the two directions may be assumed to be the vertical and horizontal ones. Since both projection polynomials are nonconstant and linear, they are permutation polynomials of $\Fp$. As $S$ does not contain a complete line, the integer values of these polynomials form the set $\{0,1,\ldots , p-1\}$. By affine changes of the two coordinates separately, we may arrange that the vertical line $x=j$ contains exactly $j$ points and the horizontal line $y=j$ contains $p-1-j$ points for $j=0,1,\ldots ,p-1$.
The column $x=0$ is empty, while the row $y=0$ contains $p-1$ points, i.e., this row contains every point $(x,0)$ with $x\ne 0$. Deleting this row and the empty column, the remaining column sums are $0,1,\ldots , p-2$ and the remaining row sums are $p-2,p-3, \ldots , 0$. Repeating the same argument shows successively that $$(x,y)\in S\Leftrightarrow y<x,$$
after the chosen affine change of coordinates. Thus $S$ is an affine image of $\mathcal T_p$, which has exactly three special directions by \cite{kisssomlai}.
\end{proof}

For larger numbers of special directions the situation is subtler. When
$k\ge4$, the projection polynomials need not all have degree exactly $k-2$. In
particular, for $k=4$ one may encounter both linear and quadratic projection
polynomials. Note that, however, in this case, by Lemma \ref{lem:2linearpolynomial} $S$ has at most one linear projection polynomial. In what follows we shall focus on this case. The key observation
is that auxiliary projection polynomials corresponding to distinct special directions
are affinely independent when taken three at a time.

\begin{lemma}\label{lem:3independent}
Let $p_1,p_2,p_3\in\Fp[z]$ be nonconstant polynomials of degrees
$1\le d_i=\deg(p_i)\le p-1$, and assume that at least two of them have degree
$\ge2$. Let $a_ix+b_iy$ $(i=1,2,3)$ be nonzero linear forms such that
the vectors $(a_i,b_i)$ are pairwise nonproportional, i.e. they represent
distinct elements of the projective line $\PP^1(\Fp)$. Then the polynomials
$$
 q_i(x,y)=p_i(a_ix+b_iy)\in \Fp[x,y], \qquad i=1,2,3,
$$
are affinely independent over $\Fp$.
\end{lemma}
\begin{proof}
Assume, toward a contradiction, that $q_1,q_2,q_3$ are affinely dependent.
That is, there exist $\lambda_1,\lambda_2,\lambda_3,\mu\in\Fp$, with
$\lambda_1,\lambda_2,\lambda_3$ not all zero, such that
$$\lambda_1q_1(x,y)+\lambda_2q_2(x,y)+\lambda_3q_3(x,y)=\mu
 \qquad\text{for all $(x,y)\in\Fp^2$.}
$$
As the total degree of each $q_i$ is at most $p-1$, this functional
identity is also a polynomial identity in $\Fp[x,y]$.
Since the constant terms can be absorbed into $\mu$, we may focus on the
homogeneous components of positive degree.

It is clear that if, in a hypothetical affine relation, the homogeneous
components of maximal degree are linearly independent, then the affine relation is
impossible. Thus, if all three polynomials have different positive degrees,
they are affinely independent. If two of them have the same degree while the
third has a different one, then it is enough to prove that the leading homogeneous
parts of the two having the same degree are linearly independent. The following
argument, with a suitable modification of the indices, works in this case as well.

More generally, in a hypothetical affine relation let
$ m=\max\{d_i:\lambda_i\ne0\}.$
It is enough to consider the homogeneous component of degree $m$. If $m=1$,
then by the assumption of the lemma there is at most one polynomial of degree
one. Hence only one $\lambda_i$ can be nonzero, which is impossible since the
corresponding $q_i$ is nonconstant. Thus $m\ge2$.

Let
$p_i(z)=c_{d_i}z^{d_i}+\cdots$ be a one-variable polynomial with leading coefficient
$c_{d_i}\ne0$. Then, for every index $i$ with $d_i=m$, the highest-degree
homogeneous part of $q_i(x,y)$ is $c_{d_i}(a_ix+b_iy)^m$.
Expanding this and treating the components as coordinates, we may write this
vector as

$$
 v_i=c_{d_i}\left(
 \binom{m}{0}a_i^m,
 \binom{m}{1}a_i^{m-1}b_i,
 \ldots,
 \binom{m}{m}b_i^m
 \right).
$$
Thus the coefficient of $x^{m-u}y^u$ is equal to $ c_{d_i}\binom{m}{u}a_i^{m-u}b_i^u$.

Taking these homogeneous components of degree $m$ in the assumed affine relation,
we obtain a linear dependence among the vectors $v_i$ corresponding to the indices for which $d_i=m$.

After a common invertible linear change of variables, we may assume that
$a_i\ne0$ for all these indices. Since $m\le p-1$, all the binomial coefficients
$\binom{m}{u}$ are nonzero in $\Fp$. Hence multiplying columns of this coefficient
matrix by nonzero scalars does not affect its rank, so the binomial coefficients
can be removed. Dividing the $i$-th row by the nonzero scalar $c_{d_i}a_i^m$,
the remaining row vectors correspond to geometric progressions with pairwise
distinct ratios $b_i/a_i$. Such geometric sequences correspond to the rows of
a Vandermonde matrix and are therefore linearly independent.

Consequently, the homogeneous component of degree $m$ in the assumed affine
relation forces
$$
 \lambda_i=0\qquad\text{for every $i$ with $d_i=m$}.
$$
This contradicts the definition of $m$. Therefore $q_1,q_2,q_3$ are affinely
independent.
\end{proof}
For $k=4$, Proposition~\ref{projections:degree} shows that every projection
polynomial has degree at most $2$. Moreover, by
Lemma~\ref{lem:2linearpolynomial}, at most one of the four special projection
polynomials can be linear. Consequently, any three of the corresponding
auxiliary projection functions contain at least two quadratic functions and
therefore satisfy the assumptions of Lemma~\ref{lem:3independent}. Hence any
three of the four auxiliary projection functions are affinely independent.

\section{The case of four special directions}
\label{sec:k4}
In this section, we prove Theorem~\ref{thm:4directions}.


 Let $S\subset \Fp^2$ have exactly four special directions. By the reduction in Section 3, we can assume that $S$ contains no complete affine line. Thus every projection value belongs to $\{0,1,\ldots , p-1\}$. By Lemma \ref{lem:2linearpolynomial}, at most one of the four special projection polynomials is linear. Therefore, any three projection polynomials satisfy the assumptions of Lemma~\ref{lem:3independent}, hence any three projection polynomials are affinely independent, i.e., every nonzero linear combination of $S_1, S_2, S_3$, regarded as a polynomial of $\Fp^2$, is nonconstant and has degree at most 2.





The following proposition is the key step in proving
Theorem~\ref{thm:4directions}. Its essence is that the values
$(S_1(x,y),S_2(x,y),S_3(x,y))$, viewed as points of $\Fp^3$, are uniformly distributed
modulo $p$. 

\begin{proposition}\label{prp:4directions-equid}
Let $S\subset\Fp^2$ have exactly four special directions, and let $S_1, S_2, S_3$ be as
above. Suppose in addition that $S$ contains no complete affine line. Then for all sufficiently large primes $p$ there exist points
$z_0,z_1\in\Fp^2$ such that
\[
 S_i(z_0)\le\frac{p}{10}
 \qquad\text{and}\qquad
 S_i(z_1)\ge\frac{9p}{10},
 \qquad i=1,2,3.
\]
\end{proposition}

\begin{proof}
Put $G=(S_1,S_2,S_3)$ and regard
\[
 M=\{G(x,y):(x,y)\in\Fp^2\}
\]
as a multiset. For
$0\ne\xi=(\xi_1,\xi_2,\xi_3)\in\Fp^3$, set
\[
 A(\xi)=\sum_{(x,y)\in\Fp^2}
 e_p\!\left(\sum_{j=1}^3\xi_jS_j(x,y)\right).
\]
By Lemma~\ref{lem:3independent}, $\sum_j\xi_jS_j$ is nonconstant and has
degree at most $2$. Lemma~\ref{lem:2.1C} therefore gives $|A(\xi)|\le p^{3/2}$.
Applying Proposition~\ref{prop: diszkrepancia szamitas} with $N=p^2$, $r=3$, and $B=p^{3/2}$, we obtain,
for every interval box $T=I_1\times I_2\times I_3$,
\[
 \left|\frac{\#G^{-1}(T)}{p^2}-\frac{|T|}{p^3}\right|
 \le\frac{L_p^3}{\sqrt p},
 \qquad L_p=1+H_{(p-1)/2}.
\]
Now let
\[
 I^- =\{0,\ldots,\lfloor p/10\rfloor\},\qquad
 I^+ =\{\lceil9p/10\rceil,\ldots,p-1\},
 \qquad B^\pm=(I^\pm)^3.
\]
Since $|I^-|\ge p/10$ and $|I^+|=\lfloor p/10\rfloor\ge(p-9)/10$,
\begin{align*}
 \#G^{-1}(B^-)&\ge \frac{p^2}{1000}-p^{3/2}L_p^3,\\
 \#G^{-1}(B^+)&\ge
 p^2\left(\frac{p-9}{10p}\right)^3-p^{3/2}L_p^3.
\end{align*}
Because $L_p=O(\log p)$, both right-hand sides are positive for all
sufficiently large $p$. Hence $M$ meets both $B^-$ and $B^+$, which is exactly
the asserted pair of inequalities.
\end{proof}

\begin{proof}[Proof of Theorem~\ref{thm:4directions}]
Suppose that $p$ is sufficiently large and $S\subseteq \Fp^2$ has exactly four special directions. As above, remove complete lines if necessary and relabel the special directions so that Proposition~\ref{prp:4directions-equid} can be applied. Let $|S|=np$. By Proposition \ref{projections:explicit},
\[
 S(x,y)= -\frac{3n}{p}
 +\frac1p\sum_{i=1}^4\bigl(S_i(x,y)\bigr).
\]
Let $z_0,z_1\in \Fp^2$ be given by Proposition \ref{prp:4directions-equid}.
Then, for $i=1,2,3$,
$$S_i(z_1)-S_i(z_0)\ge \lceil9p/10\rceil-\lfloor p/10\rfloor \ge \frac{4p}{5}.$$
Moreover, $0\le S_4(z) \le p-1$ for every $z \in \Fp^2$, and hence $$S_4(z_1)-S_4(z_0)\ge -(p-1).$$

Subtracting the explicit formula at the two chosen points therefore gives
\begin{equation*}
\begin{split}&S(z_1)-S(z_0)=\frac1p\sum_{i=1}^4\bigl(S_i(z_1)-S_i(z_0)\bigr)\\
&\ge \frac 1p \bigl( 3\cdot \frac{4p}{5}-(p-1)\bigr)= \frac{7}{5}+\frac 1p >1.
\end{split}
\end{equation*}
This is impossible, since all the values of a characteristic function are 0 and 1 and hence the left-hand side belongs to $\{-1,0,1\}$. The contradiction proves the theorem.
\end{proof}
\section{A conditional extension to more special directions}
\label{sec:affineextension}
The proof of Theorem~\ref{thm:4directions} uses three properties of the chosen
projection functions:
\begin{itemize}
    \item they are affinely independent,
    \item  their degrees are
bounded,
\end{itemize}
  In the four-direction case the relevant
polynomials have degree at most $2$, so the explicit estimate of
Lemma~\ref{lem:2.1C} is sufficient. For projection functions of arbitrary bounded
degree, Corollary~\ref{cor:weyl} yields the same qualitative
equidistribution statement. We now prove the conditional extension stated in Theorem~\ref{thm:affineextension-intro}.

First observe that all $k$ special projection functions can never be affinely
independent. Indeed, reducing the identity in
Proposition~\ref{projections:explicit} modulo $p$ gives
\[
 \sum_{i=1}^k S_i=(k-1)n.
\]
Thus there is always an affine relation among the full family.

\begin{theorem}\label{thm:affineextension}
Fix integers $k\ge4$ and $r$ such that
\[
 \frac{k+1}{2}<r\le k-1.
\]
Then there exists $T_{k,r}$ such that, for every prime $p\ge T_{k,r}$,
there is no subset $S\subseteq\Fp^2$ of cardinality divisible by $p$
with exactly $k$ special directions for which $r$ of the corresponding
auxiliary projection functions $S_i$, regarded as polynomials on
$\Fp^2$, are affinely independent.
\end{theorem}

\begin{proof}
Suppose that such a set $S$ exists. As in Section~\ref{sec:projection}, we
first remove complete affine lines and then relabel $|S|/p$ as $n$. This
reduction does not affect the affine-independence hypothesis. Indeed, in the
direction of a removed line the projection function is unchanged modulo
$p$, while in every other direction it is changed only by an additive
constant. Adding constants to the individual projection functions does not
affect their affine independence.

Relabel the special directions so that $S_1,\ldots,S_r$ are affinely
independent, and put
\[
 G=(S_1,\ldots,S_r):\Fp^2\longrightarrow\Fp^r.
\]
By Proposition~\ref{projections:degree}, every $S_i$, regarded as a
polynomial on $\Fp^2$, has degree at most $k-2$. Hence, for every
$0\ne\xi=(\xi_1,\ldots,\xi_r)\in\Fp^r$, affine independence implies that
\[
 Q_\xi=\sum_{i=1}^r\xi_iS_i
\]
is a nonconstant polynomial of degree at most $k-2$. For $p$ sufficiently
large, Corollary~\ref{cor:weyl} gives
\[
 \left|\sum_{z\in\Fp^2}e_p(Q_\xi(z))\right|
 \le C_{2,k-2}p^{3/2}.
\]
Applying Proposition~\ref{prop: diszkrepancia szamitas}, now with target dimension $r$, gives, for every
interval box $T\subseteq\Fp^r$,
\[
 \left|\frac{\#G^{-1}(T)}{p^2}-\frac{|T|}{p^r}\right|
 \le C_{2,k-2}\frac{L_p^r}{\sqrt p}=o(1),
\]
since $k$ and $r$ are fixed.

Choose $\varepsilon>0$ so small that
\[
 r(1-2\varepsilon)-(k-r)>1.
\]
Such an $\varepsilon$ exists because $2r-k>1$. Let
\[
 I^-_\varepsilon=\{0,\ldots,\lfloor\varepsilon p\rfloor\},
 \qquad
 I^+_\varepsilon=
 \{\lceil(1-\varepsilon)p\rceil,\ldots,p-1\}.
\]
Since the normalized sizes of both intervals tend to $\varepsilon$, the
preceding discrepancy estimate shows that, for all sufficiently large $p$,
there are $z_0,z_1\in\Fp^2$ such that
\[
 S_i(z_0)\le\varepsilon p,
 \qquad
 S_i(z_1)\ge(1-\varepsilon)p,
 \qquad i=1,\ldots,r.
\]
Consequently,
\[
 S_i(z_1)-S_i(z_0)\ge(1-2\varepsilon)p,
 \qquad i=1,\ldots,r.
\]
For each of the remaining $k-r$ projection functions we only use
$0\le S_i\le p-1$, and hence
\[
 S_i(z_1)-S_i(z_0)\ge-(p-1).
\]

Subtracting the identity of Proposition~\ref{projections:explicit} at
$z_0$ and $z_1$, we obtain
\begin{equation*}
 S(z_1)-S(z_0)
 \ge r(1-2\varepsilon)
      -(k-r)\left(1-\frac1p\right) =r(1-2\varepsilon)-(k-r)+\frac{k-r}{p} >1
\end{equation*}
for all sufficiently large $p$. This is impossible because
$S(z_1)-S(z_0)\in\{-1,0,1\}$.
\end{proof}

\begin{remark}
In particular, for every fixed $k\ge4$, the theorem applies whenever $k-1$
of the $k$ special projection functions are affinely independent. We do not
claim that such affine independence must hold in general. 
This result is
a conditional extension of the four-direction argument.


As a corollary, if $k=5$, then Lemma~\ref{lem:3independent} shows that
any three auxiliary projection functions are affinely independent. On the
other hand, in any hypothetical example for sufficiently large $p$, every
four of them must be affinely dependent by Theorem~\ref{thm:affineextension}. More precisely, one can verify that the only case not treated with our tools is when all 5 nonconstant projection polynomials are of degree 2.
\end{remark}
\section{A second-moment argument for few special directions}
\label{sec:second-moment}

The projection polynomial method and the reconstruction formula of
Section~\ref{sec:projection}, together with the estimates of
Section~\ref{sec:equidistribution}, also lead to a second-moment argument.
The additional observation is that projections of a uniformly chosen point
of the plane onto two distinct directions are independent. This allows us to
compute the variance of the sum of the special projection functions in two
ways, leading to a contradiction.

Throughout this section, $k\ge4$ is fixed, except in the final subsection,
where we briefly discuss the critical case $k=3$.
We retain the notation introduced in Section~\ref{sec:equidistribution}:
\[
 L_p=1+H_{(p-1)/2}=1+\sum_{r=1}^{(p-1)/2}\frac1r\le 2+\log p.
\]
\begin{theorem}\label{thm:quantitative}
Fix an integer $k\ge4$. If $p$ is a prime satisfying
\[
 p\ge \bigl(64k\log(8k)\bigr)^2.
\]
Then no subset of $\Fp^2$ has exactly $k$ special directions.
\end{theorem}

\begin{remark}
The proof gives the sharper $p$-dependent sufficient condition
$kL_p/\sqrt p\le1/8$. It also shows that, for every fixed
$c<(\sqrt{21}-3)/12$, all sufficiently large primes admit no such set when
$k\le c\sqrt p/L_p$.
\end{remark}

\subsection{One-variable discrepancy and moments}
\label{sec:one-variable}

For $a\in\Fp$, write $[a]_p$ for its standard integer representative in
$\{0,1,\ldots,p-1\}$. This identification is standard in this context, see,
for example, \cite{somlai}.
\begin{proposition}
\label{prop:poly-discrepancy}
Let $P\in\Fp[x]$ be nonconstant with $1\le d=\deg P\le D<p$, let $U$ be
uniform on $\Fp$, and set $X=[P(U)]_p$, regarded as an integer-valued random
variable. Then, for every interval $I\subseteq\{0,1,\ldots,p-1\}$,
\[
 \left|\mathbb P(X\in I)-\frac{|I|}{p}\right|
 \le \frac{(D-1)L_p}{\sqrt p}=:\delta_{D,p}.
\]
For $d=1$ the left-hand side is $0$.
\end{proposition}

\begin{proof}
In Proposition~\ref{prop: diszkrepancia szamitas}, take the finite set $X$
there to be $\Fp$, and take $N=p$, $r=1$, $G=P$, and $T=I$. By
Theorem~\ref{thm:weil}, for every $h\in\Fp^\times$,
\[
 \left|\sum_{t\in\Fp}e_p(hP(t))\right|\le(d-1)\sqrt p,
\]
so the parameter $B$ in Proposition~\ref{prop: diszkrepancia szamitas} may
be taken to be $(d-1)\sqrt p$. Hence
\[
 \left|\mathbb P(X\in I)-\frac{|I|}{p}\right|
 \le\frac{(d-1)L_p}{\sqrt p}\le\delta_{D,p}.
\]
If $d=1$, then $P$ is a permutation of $\Fp$, so the discrepancy is zero.
\end{proof}

We next convert interval discrepancy into moment estimates. Let $U_0$ be
uniform on $\Fp$, so that $\mathbb P(U_0=a)=1/p$ for every $a\in\Fp$.
Then $[U_0]_p$ is an integer-valued random variable, uniformly distributed on
$\{0,1,\ldots,p-1\}$. Define
\[
 Z_p:=\frac{[U_0]_p}{p}.
\]
Thus $Z_p$ is real-valued, and all expectations and variances below are taken
over the real numbers. A direct calculation gives
\begin{equation}\label{eq:uniform-moments}
 \mathbb E Z_p=\frac{p-1}{2p},\qquad
 \Var(Z_p)=\frac{p^2-1}{12p^2}.
\end{equation}
\begin{proposition}\label{prop:poly-variance}
Let $1\le\deg P\le D<p$, let $U$ be uniform on $\Fp$, and put
$X=[P(U)]_p$. Then
\[
 \left|\mathbb E\frac Xp-\mathbb E Z_p\right|
 \le \left(1-\frac1p\right)\delta_{D,p},
\]
and
\[
 \left|\Var\!\left(\frac Xp\right)-\Var(Z_p)\right|
 \le \left(1-\frac1p\right)^2
 \left(\frac{\delta_{D,p}}2+\delta_{D,p}^2\right).
\]
In particular, for every fixed $D$,
$\Var(X)=(1/12+o_D(1))p^2$ for all nonconstant polynomials of
degree at most $D$.
\end{proposition}

\begin{proof}
Write
\[
 \mu_j=\mathbb P(X=j),\qquad \nu_j=\frac1p,\qquad
 \eta_j=\mu_j-\nu_j\qquad(0\le j\le p-1),
\]
and, for $0\le r\le p-1$, put $A_r=\sum_{j=0}^r\eta_j$. For the initial
interval $I_r=\{0,1,\ldots,r\}$ we have
\[
 A_r=\mathbb P(X\in I_r)-\frac{|I_r|}{p}.
\]
Thus Proposition~\ref{prop:poly-discrepancy} gives
$|A_r|\le\delta_{D,p}$ for every $r$, while $A_{p-1}=0$ since
$\sum_{j=0}^{p-1}\mu_j=\sum_{j=0}^{p-1}\nu_j=1$.

We now express a weighted sum of the $\eta_j$ in terms of these partial
sums. Since $\eta_0=A_0$ and $\eta_j=A_j-A_{j-1}$ for $j\ge1$, for any real
sequence $b_0,\ldots,b_{p-1}$,
\[
\begin{aligned}
 \sum_{j=0}^{p-1}b_j\eta_j
 &=b_0A_0+\sum_{j=1}^{p-1}b_j(A_j-A_{j-1})=\sum_{j=0}^{p-1}b_jA_j-\sum_{j=0}^{p-2}b_{j+1}A_j\\
 &=-\sum_{j=0}^{p-2}A_j(b_{j+1}-b_j),
\end{aligned}
\]
where the last equality uses $A_{p-1}=0$. Therefore
\begin{equation}\label{eq:variation}
 \left|\sum_{j=0}^{p-1}b_j(\mu_j-\nu_j)\right|
 \le\delta_{D,p}\sum_{j=0}^{p-2}|b_{j+1}-b_j|.
\end{equation}

Applying \eqref{eq:variation} with $b_j=j/p$ gives
\[
 \left|\mathbb E\frac Xp-\mathbb E Z_p\right|
 \le \left(1-\frac1p\right)\delta_{D,p},
\]
since $\sum_{j=0}^{p-2}|b_{j+1}-b_j|=(p-1)/p$.

Next, put $m=\mathbb E Z_p$ and apply \eqref{eq:variation} with
$b_j=(j/p-m)^2$. In this case,
\[
 \sum_{j=0}^{p-1}b_j\mu_j=\mathbb E(X/p-m)^2,
 \qquad
 \sum_{j=0}^{p-1}b_j\nu_j=\Var(Z_p).
\]
Since $p$ is odd, the sequence $b_j$ decreases from $m^2$ to $0$
and then increases back to $m^2$, so
\[
 \sum_{j=0}^{p-2}|b_{j+1}-b_j|
 =2m^2=\frac12\left(1-\frac1p\right)^2.
\]
Thus \eqref{eq:variation} yields
\[
 \left|\mathbb E(X/p-m)^2-\Var(Z_p)\right|
 \le \frac12\left(1-\frac1p\right)^2\delta_{D,p}.
\]

Finally,
\[
 \Var(X/p)=\mathbb E(X/p-m)^2-\bigl(\mathbb E(X/p)-m\bigr)^2,
\]
so the two preceding estimates give the stated variance bound.
\end{proof}
In the proof of
Theorem~\ref{thm:quantitative}, we only need the lower bound of
Proposition~\ref{prop:poly-variance} so we  state it  separately.
\begin{smcorollary}\label{cor:poly-variance-lower}
Under the assumptions of Proposition~\ref{prop:poly-variance},
\[
 \frac{\Var(X)}{p^2}\ge \frac{p^2-1}{12p^2}
 -\left(1-\frac1p\right)^2
 \left(\frac{\delta_{D,p}}2+\delta_{D,p}^2\right).
\]
\end{smcorollary}

\begin{proof}
This is the lower half of the variance estimate in
Proposition~\ref{prop:poly-variance}, together with
\eqref{eq:uniform-moments}.
\end{proof}

\subsection{Projection functions and pairwise independence}

Suppose that $S\subseteq\Fp^2$ has exactly $k$ special directions and
$k<p+1$. Then a non-special direction exists, so $|S|$ is divisible by $p$.
As in Section~\ref{sec:projection}, remove complete affine lines if necessary;
for $k\ge4$ this does not change the set of special directions. We continue
to write $S$ for the reduced set and write $|S|=np$.

Let $\mathbf m_1,\ldots,\mathbf m_k$ be the special directions, and retain
the notation $a_i$ and $S_i$ from Section~\ref{sec:projection}. Thus, writing
$\ell_i(x,y)=m_{i,x}x+m_{i,y}y$, we have
\[
 a_i(t)=\#\{z\in S:\ell_i(z)=t\},\qquad S_i(z)=a_i(\ell_i(z)).
\]
Since the reduced set contains no complete affine line,
$0\le a_i(t)\le p-1$. Let $P_i\in\Fp[t]$ be the projection polynomial
representing $a_i$. Proposition~\ref{projections:degree} gives
\begin{equation}\label{eq:degree-range}
 1\le\deg P_i\le k-2,\qquad a_i(t)=[P_i(t)]_p.
\end{equation}

The reconstruction identity \eqref{eq4.3} becomes
\begin{equation}\label{eq:sum}
 \sum_{i=1}^kS_i(z)=(k-1)n+p\one_S(z).
\end{equation}

\begin{lemma}[Pairwise independence]\label{lem:pairwise}
Let $Z$ be uniform on $\Fp^2$ and let $i\ne j$. Then
$V_i=\ell_i(Z)$ and $V_j=\ell_j(Z)$ are independent and uniform on $\Fp$.
Consequently, $X_i=S_i(Z)$ and $X_j=S_j(Z)$ are independent.
\end{lemma}

\begin{proof}
The kernels of $\ell_i$ and $\ell_j$ are distinct, so the two linear forms
are linearly independent. Hence $z\mapsto(\ell_i(z),\ell_j(z))$ is a
bijection of $\Fp^2$. Therefore, for every $a,b\in\Fp$,
\[
 \mathbb P(V_i=a,V_j=b)=\frac1{p^2}
 =\mathbb P(V_i=a)\mathbb P(V_j=b).
\]
Thus $V_i$ and $V_j$ are independent and uniform. Since $X_i=a_i(V_i)$ and
$X_j=a_j(V_j)$ are functions of independent random variables, they are
independent as well.
\end{proof}

\begin{remark}
Only pairwise independence is used. Moreover, three projections need not be jointly
independent. For example, if $Z=(X,Y)$ is uniform on $\Fp^2$, then
$X,Y,X+Y$ are pairwise independent but not jointly independent, which is by the way an important step towards understanding the construction of Kiss and Somlai \cite{kisssomlai} of a set having exactly 3 special directions.
\end{remark}
\subsection{Estimating the variance}
Let $Z$ be uniform on $\Fp^2$ and let
$X_i=S_i(Z)$ and $W=\sum_{i=1}^kX_i$.
Here each $X_i$ takes its values in $\{0,1,\ldots,p-1\}$, as in
\eqref{eq:degree-range}, and $W$ is the sum of these integer values,
without reduction modulo $p$. All expectations and variances below
are taken over $\mathbb R$.

Since $\ell_i(Z)$ is uniform on $\Fp$,
\begin{equation}\label{eq:projection-mean}
 \mathbb E X_i=\frac1p\sum_{t\in\Fp}a_i(t)=n.
\end{equation}
For $k-2<p$, let
$ \delta:=\frac{(k-3)L_p}{\sqrt p}$.
Applying the first estimate of Proposition~\ref{prop:poly-variance} to
$P_i$ and using \eqref{eq:uniform-moments} gives
\begin{equation}\label{eq:density-estimate}
 \left|\frac np-\frac{p-1}{2p}\right|
 \le\left(1-\frac1p\right)\delta.
\end{equation}
In particular, for fixed $k$, the density $|S|/p^2=n/p$ tends to $1/2$ as $p$ tends to infinity.

By Lemma~\ref{lem:pairwise}, the covariances of distinct $X_i$ vanish.
Together with \eqref{eq:sum}, this gives
\begin{equation}\label{eq:variance-upper}
 \Var(W)=\sum_{i=1}^k\Var(X_i)
 =p^2\Var(\one_S(Z))=n(p-n)\le\frac{p^2}{4}.
\end{equation}
\subsection{Proof of Theorem~\ref{thm:quantitative}}

\begin{proof}
Suppose that such a set $S$ exists. The assumed bound on $p$ implies
$p\ge11$ and $k<p+1$, so the preceding reduction and notation apply. Put
\[
 x=\frac{kL_p}{\sqrt p},\qquad
 \delta=\frac{(k-3)L_p}{\sqrt p}=\left(1-\frac3k\right)x.
\]
Corollary~\ref{cor:poly-variance-lower}, applied with $D=k-2$, and
\eqref{eq:variance-upper} give
\[
 \frac14\ge k\left[\frac{p^2-1}{12p^2}
 -\left(1-\frac1p\right)^2\left(\frac\delta2+\delta^2\right)\right].
\]
Since $(1-1/p)^2\le1$, this implies
\begin{equation}\label{eq:main-comparison}
 \frac{k-3}{12}-\frac{k}{12p^2}
 \le (k-3)\left(\frac x2+x^2\right).
\end{equation}
If $x\le1/8$, then the right-hand side is at most
$5(k-3)/64$. Hence \eqref{eq:main-comparison} would imply
\[
 \frac{k-3}{192}\le\frac{k}{12p^2},
\]
which is impossible for $p\ge11$, since $k/(k-3)\le4$.

It remains to verify $x\le1/8$. The function
$(2+\log t)/\sqrt t$ is decreasing for $t>1$, and $L_p\le2+\log p$. Since
$p\ge(64k\log(8k))^2$,
\[
 x\le\frac{1+\log(64k\log(8k))}{32\log(8k)}\le\frac18.
\]
For the last inequality we used $\log(8k)\le k$ and
$1+\log(64k^2)\le4\log(8k)$ for $k\ge4$. This contradiction proves the
theorem.
\end{proof}

\subsection{The critical case $k=3$}

The comparison is sharp at the level of leading constants. For a
nonconstant polynomial of fixed degree, Proposition~\ref{prop:poly-variance}
gives variance $(1/12+o(1))p^2$. Thus $k$ special projections contribute
asymptotically $(k/12)p^2$, whereas \eqref{eq:variance-upper} is at most
$p^2/4$. These constants coincide when $k=3$.

In fact, equality holds exactly for the triangle
\[
 \mathcal T_p=\{(x,y)\in\Fp^2:0\le y<x\le p-1\}.
\]
Here $|\mathcal T_p|=p(p-1)/2$, so $n=(p-1)/2$, and the three special
projection polynomials are linear permutation polynomials. Hence each
corresponding random variable has variance $(p^2-1)/12$, and
\[
 \sum_{i=1}^3\Var(X_i)=3\frac{p^2-1}{12}
 =\frac{p^2-1}{4}=n(p-n).
\]
Thus the two variance calculations agree exactly for $k=3$.

\section{Multisets with prescribed special directions}
\label{sec:multisets}

In this section we consider multisets on $\Fp^2$, identified with
nonnegative integer-valued functions
$$ f:\Fp^2\longrightarrow \mathbb Z_{\ge 0}.$$
For such a function, a direction is called non-special if the sum of $f$
along every line in that direction is the same, and special otherwise.

We show that, in contrast with the situation for sets, an arbitrary collection
of at most $p$ directions can occur as the set of special directions of a
multiset taking only the values $0,1,2$.

\begin{theorem}\label{thm:prescribed-multiset-directions}
Let $p>2$ be a prime, and let
$ \mathcal D\subseteq\PP^1(\Fp) $
be any set of directions with $|\mathcal D|\le p$. Then there exists a
multiset
$$
 f:\Fp^2\longrightarrow\{0,1,2\}
$$
whose set of special directions is exactly $\mathcal D$.
\end{theorem}

\begin{proof}
We first construct such multisets when the vertical direction does not belong
to $\mathcal D$. Identifying the nonvertical directions with their slopes in
$\Fp$, write
$$
 \mathcal D=A\subseteq\Fp.
$$
For $a\in\Fp$, let
$$
 L_a=\{(x,y)\in\Fp^2:y=a^2+ax\}.
$$
The lines $L_a$, $a\in\Fp$, have pairwise distinct slopes. Define
$$
 f_A=\sum_{a\in A}\one_{L_a}.
$$

We first determine the special directions of $f_A$. Let $k=|A|$ and fix $a\in A$. Along the line $L_a$ itself the contribution of $\one_{L_a}$ is $p$, while every other line $L_b (b\in A\setminus \{a\})$, intersects $L_a$ in exactly one point.
Hence
$$
\sum_{z\in L_a} f_A(z)=p+k-1.
$$
On the other hand, every line parallel to $L_a$ and distinct from $L_a$ is disjoint from $L_a$ and intersects each  $L_b$, in exactly one point for every $ b\in A\setminus \{a\}$. Thus its total multiplicity is $k-1$.
Therefore the direction of slope $a\in A$ is special.

Now let $c\in \Fp\setminus A$. Every line of slope $c$ intersects each $L_a, a\in A$, in exactly one point, and consequently every such line has total multiplicity $k$. Hence, the direction of slope $c$ is non-special. The same is true for the vertical direction. It follows that the set of special directions of $f_A$ is exactly $A$.

It remains to show that $f_A$ takes only values $0,1,2$. For a fixed point $(x,y)\in \Fp^2$, the lines $L_a, a\in A$ passing through $(x,y)$ are precisely those for which
$$
y=a^2+ax,
$$
or equivalently,
$a^2+xa-y=0.$ This is a quadratic equation in $a$, and hence has at most two solutions in $\Fp$.
Therefore, $$f_A(x,y)\in \{0,1,2\} \qquad \textrm{ for every } (x,y)\in \Fp^2.$$

Finally, let $\mathcal D\subseteq \mathbb{P}^1(\Fp)$ be arbitrary with $|\mathcal{D}|\le p$. Since $\mathbb{P}^1(\Fp)$ has $p+1$  directions, we can choose a direction $d_0\not\in \mathcal{D}$.
There is an invertible linear transformation of $\Fp^2$ sending $d_0$ to the vertical direction. Under this transformation, $\mathcal{D}$ is mapped to a set of non-vertical directions, so the previous construction can be applied. If we apply the inverse transformation to the corresponding multiset, then it preserves the pointwise multiplicity and permutes the directions so that we obtain a multiset with values from $\{0,1,2\}$ and whose special directions are exactly $\mathcal{D}$.
\end{proof}

\begin{remark}
   The family of lines used in the proof has a geometric interpretation. The lines $L_a: ~ y=a^2+ax$ are precisely the tangent lines to the parabola $y=-x^2/4$.
Thus, the fact that the multiset constructed above is $\{0,1,2\}$-valued is consistent with the fact that through a point there are at most two tangents to a nondegenerate conic. Moreover, the derivative of a parabola is a linear function which is a permutation polynomial so we have a tangent for every possible slope.
\end{remark}
\begin{corollary}
For every prime $p>2$ and every integer $k$ with
\[
 0\le k\le p,
\]
there exists a $\{0,1,2\}$-valued multiset on $\Fp^2$ having exactly
$k$ special directions.
\end{corollary}

\begin{proof}
Choose any set $\mathcal D\subseteq\PP^1(\Fp)$ of cardinality $k$ and apply
Theorem~\ref{thm:prescribed-multiset-directions}.
\end{proof}

The construction above necessarily leaves one direction non-special, and
therefore realizes at most $p$ special directions. The remaining case, in
which all $p+1$ directions are special, can also occur, in fact, for $p\ge5$ it can already be realized by a set of cardinality $p$ as follows.
\newline
Let $S$ be the following subset of $\Fp^2$ of size $p$
\[S:=\left\{ (0,i)\in \Fp^2 \mid 1\le i \le p-2 \right\} \cup
\{(1,p-1)\} \cup \{(a,p-1)\}, \quad a \ne 0,1.
\]
If $p>3$, then $S$ determines the vertical direction and for every odd prime $p$ it determines the directions of slope $i$ for $1\le i\le p-2$. The direction of slope $0$ is determined by the pair $(1,p-1)$ and $(a,p-1)$. Now the direction of slope $p-1$ is determined if and only if $-a\equiv p-1-i \pmod{p}$  for some $1\le i \le p-2$. Since $p\ge 5$, we may choose $a$ and $i$ for this equation to be satisfied.
Clearly, $S$ is a set of size $p$ determining every direction if $p\ge 5$.
\newline
For $p=3$ this is impossible for a set of cardinality $p$, since three
points determine at most $\binom{3}{2}=3$ directions. Hence every
three-point subset of $\F_3^2$ misses a direction and consequently tiles
$\F_3^2$ by a line in that direction.

\medskip
The following question arises here. Let $G$ be a finite abelian group. We say that two elements $\psi_1,\psi_2$ of $\hat{G}$ are equivalent if and only if $\langle \psi_1 \rangle=\langle \psi_2 \rangle$.
Let $\psi_1,\psi_2, \ldots ,\psi_k$ be a set of representatives of the equivalence classes of $\widehat{G}$, where $\psi_1$ is the  trivial character.

 Let $f \colon G \to \mathbb{Z}_{\ge0}$. Since $f$ is integer-valued, the support of $\widehat{f}$ is the union of equivalence classes of $\widehat{G}$. We identify its nontrivial part with a subset of $[k]\setminus\{1\}$, where $[k]=\{1,2,\ldots, k\}$. For every nonempty subset $C \subseteq [k]\setminus \{1\}$, let $m_C$ be the smallest integer $m$ such that there is a function $f \colon G \to \{0,1, \dots,m\}$ whose nontrivial support consists exactly of the equivalence classes indexed by $C$. In other words,
 $$\widehat f(\psi_i)\ne0 \quad\Longleftrightarrow\quad i\in C, \qquad i=2,\ldots,k. $$
\begin{definition}
For a finite abelian group we define the \emph{height of $G$} as follows.
$$h(G):=\max_{\emptyset \ne C \subset [k]\setminus \{1\} } m_C.$$
\end{definition}

\begin{problem}
We suggest determining the height of finite abelian groups.
    \end{problem}
We conjecture that the height of a cyclic group of order $n$ is  at most 2 if the prime divisors of $n$ are large enough compared to the number of prime divisors of $n$.  Note that $h(\mathbb{Z}_p)=1$.

 Theorem~\ref{thm:prescribed-multiset-directions} shows that
$h(\Fp^2)\le2$. Indeed, every proper nonempty collection of directions can
be realized by a $\{0,1,2\}$-valued multiset, while the collection of all
$p+1$ directions is realized by 
characteristic function of a single point.
Since no set can have exactly two special directions, whereas
Theorem~\ref{thm:prescribed-multiset-directions} gives a
$\{0,1,2\}$-valued multiset with exactly two special directions, we obtain
$h(\Fp^2)=2$.

We claim that the same holds for $\mathbb Z_{pq}$.
\begin{proposition}
$h(\mathbb{Z}_{pq})=2$ if $p$ and $q$ are different primes.
\end{proposition}
\begin{proof}
It follows from \cite{redeiegyseggyok} that a function $f\colon \mathbb{Z}_{pq} \to \mathbb{Z}_{\ge 0}$ satisfies $\hat{f}(x)=0$ whenever $x$ is of order $pq$ if and only if $f$ is a nonnegative integer linear combination of characteristic functions of translates of $\{0,p,2p,\ldots ,(q-1)p\}$ and $\{0,q,2q,\ldots ,(p-1)q\}$. We can call these a $q$-line and a $p$-line, respectively.

Suppose first that $C$ consists only of the equivalence class of elements
of order $pq$. If $f$ were $\{0,1\}$-valued, then the vanishing of
$\widehat f$ on the elements of orders $p$ and $q$ would imply that the
total sum of the values of $f$ is divisible by both $p$ and $q$, and hence
by $pq$. Thus a nonzero such $f$ would have to be identically $1$, which is
impossible since then all nontrivial Fourier coefficients vanish. Hence
$m_C\ge2$. On the other hand, under the identification
$\mathbb Z_{pq}\cong\mathbb Z_p\times\mathbb Z_q$, let
$$
f_0=\mathbf1_{\mathbb Z_{pq}}
+\mathbf1_{(0,0)}+\mathbf1_{(1,1)}
-\mathbf1_{(0,1)}-\mathbf1_{(1,0)}.
$$
Then $f_0$ is a $\{0,1,2\}$-valued multiset whose nontrivial Fourier support
consists exactly of the elements of order $pq$. Hence $m_C=2$ in this case.

Now suppose that $C$ contains the equivalence classes of elements of orders
$p$ and $q$, but not that of order $pq$. Since $\widehat f$ vanishes on the
elements of order $pq$, the decomposition above applies. If $f$ were
$\{0,1\}$-valued, then both a $p$-line and a $q$-line could not occur with
positive coefficient in this decomposition, since every $p$-line intersects
every $q$-line and the value at their intersection would be at least $2$.
Thus the Fourier support of $f$ could not contain both the order-$p$ and
the order-$q$ classes. Hence $m_C\ge2$. Conversely, the sum of the
characteristic functions of one $p$-line and one $q$-line is
$\{0,1,2\}$-valued and has nontrivial Fourier support exactly on these two
classes. Therefore $m_C=2$ also in this case.

We leave it to the reader to verify that $m_C=1$ for every other subset $C$ of $[4] \setminus \{1\}$.
\end{proof}

\end{document}